\documentclass[12pt,a4paper]{article}
\usepackage{graphicx}
\usepackage{amsfonts}
\usepackage{amsmath}
\usepackage{amssymb}
\usepackage{amsthm}
\usepackage[ansinew]{inputenc}
\usepackage{latexsym}
\usepackage{tabularx}
\usepackage{enumitem}
\allowdisplaybreaks
\newtheorem{thm}{Theorem}

\newtheorem{prop}{Proposition}
\newtheorem{defn}{Definition}
\newtheorem{rem}{Remark}

\newcommand{\abs}[1]{\left\vert#1\right\vert}

\newcommand{\To}{\rightarrow}

\newcommand{\calP}{\mathcal{P}}

\newcommand{\bsx}{\boldsymbol{x}}
\newcommand{\bsz}{\boldsymbol{z}}
\newcommand{\bse}{\boldsymbol{e}}

\newcommand{\bsr}{\boldsymbol{r}}
\newcommand{\bsc}{\boldsymbol{c}}

\newcommand{\icomp}{\mathtt{i}}

\newcommand{\Field}{\mathbb{F}}

\newcommand{\NN}{\mathbb{N}}
\newcommand{\ZZ}{\mathbb{Z}}

\newcommand{\CC}{\mathbb{C}}
\newcommand{\FF}{\mathbb{F}}

\newcommand{\oa}{{\rm OA}}
\newcommand{\ooa}{{\rm OOA}}
\newcommand{\prof}{{\rm PROFILE}}
\newcommand{\hei}{{\rm HEIGHT}}

\makeatletter
\newcommand{\rdots}{\mathinner{\mkern1mu\lower-1\p@\vbox{\kern7\p@\hbox{.}}
\mkern2mu \raise4\p@\hbox{.}\mkern2mu\raise7\p@\hbox{.}\mkern1mu}}
\makeatother

\date{\today}

\begin{document}

\title{Mixed orthogonal arrays, $(u,m,\bse,s)$-nets, and $(u,\bse,s)$-sequences}

\author{Peter Kritzer\thanks{P. Kritzer is supported by the Austrian Science Fund (FWF)
Project F5506-N26, which is a part of the Special Research Program
"Quasi-Monte Carlo Methods: Theory and Applications".} and Harald
Niederreiter}

\maketitle

\begin{abstract}
We study the classes of $(u,m,\bse,s)$-nets and $(u,\bse,s)$-sequences, which 
are generalizations of $(u,m,s)$-nets and $(u,s)$-sequences, respectively.  
We show equivalence results that link the existence of $(u,m,\bse,s)$-nets and so-called 
mixed (ordered) orthogonal arrays, thereby generalizing earlier results by Lawrence, and Mullen and Schmid. 
We use this combinatorial equivalence principle to obtain new results on the possible parameter 
configurations of $(u,m,\bse,s)$-nets and $(u,\bse,s)$-sequences, which generalize in particular a result of Martin and Stinson. 
\end{abstract}




\section{Introduction and basic definitions} \label{secintr}

The construction of point sets and sequences with good
equidistribution properties is a classical problem in number
theory and has important applications to quasi-Monte Carlo methods
in numerical analysis (see the books of Dick and
Pillichshammer~\cite{DP10}, Leobacher and
Pillichshammer~\cite{LP14}, and Niederreiter~\cite{N92}). The
standard setting is that of the $s$-dimensional unit cube
$[0,1]^s$, for a given dimension $s \ge 1$, from which the points
are taken. While the problem of constructing evenly distributed points in $[0,1]^s$ is of number-theoretic origin, it also has a strong combinatorial 
flavor (see \cite[Chapter~6]{DP10} and \cite[Chapter~15]{LM98}).

Powerful methods for the construction of finite point sets with
good equidistribution properties in $[0,1]^s$ are based on the
theory of nets (see again the references above as well as the
original paper~\cite{N87} and the recent handbook
article~\cite{N13}). This theory was recently extended by
Tezuka~\cite{T13} and studied in a slightly modified form by
Hofer~\cite{H15}, Hofer and Niederreiter~\cite{HN13}, Kritzer and
Niederreiter~\cite{KN14}, and Niederreiter and Yeo~\cite{NY13}.
The underlying idea of these nets is to guarantee perfect
equidistribution of the points for certain subintervals of the
half-open unit cube $[0,1)^s$. Concretely, for a dimension $s \ge
1$ and an integer $b \ge 2$, an interval $J \subseteq [0,1)^s$ is
called an \emph{elementary interval in base} $b$ if it is of the
form
\begin{equation} \label{eqelint}
J=\prod_{i=1}^s [a_ib^{-d_i},(a_i+1)b^{-d_i})
\end{equation}
with integers $d_i \ge 0$ and $0 \le a_i < b^{d_i}$ for $1 \le i
\le s$. These intervals play a crucial role in the subsequent
definition of a $(u,m,\bse,s)$-net, which we state below. Here and
in the following, we denote by $\NN$ the set of positive integers
and by $\lambda_s$ the $s$-dimensional Lebesgue measure.

\begin{defn} \label{defnet} {\rm
Let $b \ge 2$, $s \ge 1$, and $0 \le u \le m$ be integers and let
$\bse =(e_1,\ldots,e_s) \in \NN^s$. A point set $\calP$ of $b^m$
points in $[0,1)^s$ is a $(u,m,\bse,s)$-\emph{net in base} $b$ if
every elementary interval $J \subseteq [0,1)^s$ in base $b$ of
volume $\lambda_s(J) \ge b^{u-m}$ and of the form~\eqref{eqelint},
with integers $d_i \ge 0$, $0 \le a_i < b^{d_i}$, and $e_i | d_i$
for $1 \le i \le s$, contains exactly $b^m \lambda_s(J)$ points of
$\calP$. }
\end{defn}

Definition~\ref{defnet} is the definition of a $(u,m,\bse,s)$-net
in base $b$ in the sense of~\cite{HN13}. Previously,
Tezuka~\cite{T13} introduced a slightly more general definition
where the conditions on the number of points in the elementary
intervals need to hold only for those elementary intervals $J$ in
base $b$ with $\lambda_s(J)=b^{u-m}$. The narrower definition
in~\cite{HN13} guarantees, as stated in that paper, that every
$(u,m,\bse,s)$-net in base $b$ is also a $(v,m,\bse,s)$-net in
base $b$ for every integer $v$ with $u \le v \le m$. The latter
property is very useful when working with such point sets (see
again~\cite{HN13} for further details). Hence, whenever we speak
of a $(u,m,\bse,s)$-net here, we mean a $(u,m,\bse,s)$-net in the
narrower sense of Definition~\ref{defnet}.

Note that the points of a $(u,m,\bse,s)$-net tend to be very
evenly distributed if $u$ is small. But the choice of
$e_1,\ldots,e_s \in \NN$ also plays an important role since larger
values of the $e_i$ in general entail fewer restrictions in the
defining property of the net.

For infinite sequences of points in $[0,1]^s$ with good
equidistribution properties, the corresponding concept is that of
a $(u,\bse,s)$-sequence. As usual, we write $[\bsx]_{b,m}$ for the
coordinatewise $m$-digit truncation in base $b$ of $\bsx \in
[0,1]^s$ (compare with \cite[Remark 14.8.45]{N13} and
\cite[p.~194]{NX01}).

\begin{defn} \label{defseq} {\rm
Let $b \ge 2$, $s \ge 1$, and $u \ge 0$ be integers and let $\bse
\in \NN^s$. A sequence $\bsx_1,\bsx_2,\ldots$ of points in
$[0,1]^s$ is a $(u,\bse,s)$-\emph{sequence in base} $b$ if for all
integers $g \ge 0$ and $m > u$, the points $[\bsx_n]_{b,m}$ with
$gb^m < n \le (g+1)b^m$ form a $(u,m,\bse,s)$-net in base $b$. }
\end{defn}

Again, the points of a $(u,\bse,s)$-sequence are very evenly
distributed if $u$ is small, but also in this case the choice of
$\bse$ has an influence on the manner in which the points are
spread over the elementary intervals in the unit cube.

If we choose $\bse =(1,\ldots,1) \in \NN^s$ in
Definitions~\ref{defnet} and~\ref{defseq}, then these definitions
coincide with those of a classical $(u,m,s)$-net and a classical
$(u,s)$-sequence, respectively. The reasons why the more general
$(u,m,\bse,s)$-nets and $(u,\bse,s)$-sequences were introduced
have to do with their applications to quasi-Monte Carlo methods.
Since this paper is devoted to the combinatorial aspects of
$(u,m,\bse,s)$-nets and $(u,\bse,s)$-sequences, we do not
elaborate on these reasons and we refer instead to
\cite[Section~1]{KN14} and~\cite{T13}.

It was shown by Lawrence~\cite{L96} and Mullen and
Schmid~\cite{MS96} that classical $(u,m,s)$-nets are
combinatorially equivalent to certain types of orthogonal arrays
(see also \cite[Section~6.2]{DP10} for an exposition of this
result). This equivalence has important implications for the
theory of $(u,m,s)$-nets and $(u,s)$-sequences (see
\cite[Chapter~6]{DP10} and~\cite{SS09}). The main result of the
present paper generalizes this equivalence to $(u,m,\bse,s)$-nets
(see Theorem~\ref{thmequiv}). The crucial step is to move from
orthogonal arrays to mixed orthogonal arrays in the sense of
\cite[Chapter~9]{HSS99}. We recall the definition of a mixed
orthogonal array $\oa \left(N,l_1^{k_1} \cdots l_v^{k_v},t\right)$
from \cite[Definition~9.1]{HSS99}, where we change the notation
from $s_i$ to $l_i$ since in our case $s$ stands for a dimension.
We write $R(b)=\{0,1,\ldots,b-1\} \subset \ZZ$ for every integer
$b \ge 2$.

\begin{defn} \label{defmoa} {\rm
Let $N \ge 1$, $l_1,\ldots,l_v \ge 2$, $k_1,\ldots,k_v \ge 1$, and
$0 \le t \le k := k_1 + \cdots + k_v$ be integers. A \emph{mixed
orthogonal array} $\oa \left(N,l_1^{k_1} \cdots
l_v^{k_v},t\right)$ is an array of size $N \times k$ in which the
first $k_1$ columns have symbols from $R(l_1)$, the next $k_2$
columns have symbols from $R(l_2)$, and so on, with the property
that in any $N \times t$ subarray every possible $t$-tuple occurs
an equal number of times as a row. }
\end{defn}

\begin{rem} \label{remmoa} {\rm
The parameter $t$ of a mixed orthogonal array is called its
\emph{strength}. Definition~\ref{defmoa} is vacuously satisfied
for $t=0$. As in \cite[Definition~9.1]{HSS99}, it is not required
that $l_1,\ldots,l_v$ be distinct. If $l_1= \cdots =l_v$, then
Definition~\ref{defmoa} reduces to that of an orthogonal array
(see \cite[Definition~1.1]{HSS99}). }
\end{rem}

Further results of this paper concern bounds on the parameters of
$(u,m,\bse,s)$-nets and $(u,\bse,s)$-sequences for the case of
greatest practical interest where $u=0$ (see Theorems~\ref{thmnetoa} to~\ref{thmkr}). 
Moreover, we show a necessary condition for the parameters 
of a mixed ordered orthogonal array (see Theorem~\ref{thmcalD}) which generalizes \cite[Lemma~3.1]{MS99}.

\section{Necessary conditions for $(0,m,\bse,s)$-nets} \label{secnecnet}

The parameter $u$ of a $(u,m,\bse,s)$-net is a nonnegative integer
and its optimal value is $u=0$. The following result imposes a
combinatorial obstruction on the existence of $(0,m,\bse,s)$-nets.
If $\bse =(e_1,\ldots,e_s) \in \NN^s$, then we can assume without
loss of generality that $e_1 \le e_2 \le \cdots \le e_s$.

\begin{thm} \label{thmnetoa}
Let $\bse =(e_1,\ldots,e_s) \in \NN^s$ with $e_1 \le e_2 \le
\cdots \le e_s$. For $2 \le t \le s$ and $m \ge e_{s-t+1} + \cdots
+ e_{s-1}+e_s$, the existence of a $(0,m,\bse,s)$-net in base $b$
implies the existence of a mixed orthogonal array $\oa
\left(b^m,l_1^1 \cdots l_s^1,t\right)$ with $l_i=b^{e_i}$ for $1
\le i \le s$.
\end{thm}

\begin{proof}
Let $\calP$ be a $(0,m,\bse,s)$-net in base $b$ and let the points
of $\calP$ be
$$
\bsx_n=(x_n^{(1)},\ldots,x_n^{(s)})\in [0,1)^s \qquad \mbox{for }
n=1,\ldots,b^m.
$$
Furthermore, define
$$
z_i (n)=\lfloor b^{e_i} x_n^{(i)}\rfloor\in R(b^{e_i}) \qquad
\mbox{for } 1\le n\le b^m, \ 1\le i\le s.
$$
Arrange these integers into the $b^m \times s$ array
\begin{equation}\label{eqarray}
(z_i (n))_{1\le n\le b^m, \, 1\le i\le s}=\begin{pmatrix}
                                           z_1 (1) & z_2 (1) &\ldots & z_s (1)\\
                       z_1 (2) & z_2 (2) &\ldots & z_s (2)\\
                       \vdots &\vdots & &\vdots\\
                       z_1 (b^m) & z_2 (b^m) &\ldots & z_s (b^m)
                                          \end{pmatrix}.
\end{equation}
For $i=1,\ldots,s$, let $\bsz_i$ denote the $i$th column of the array in~\eqref{eqarray}.
Choose a strength $t$ with $2\le t\le s$ and assume that $m \ge
e_{s-t+1} + \cdots +e_{s-1}+e_s$, i.e., that $m$ is at least as
large as the sum of the $t$ largest $e_i$. Pick $1 \le i_1 <i_2 <
\cdots< i_t \le s$ and consider the corresponding columns
$\bsz_{i_1},\ldots,\bsz_{i_t}$. We have to show that $\bsz_{i_1},\ldots,\bsz_{i_t}$ are orthogonal in
the sense of Definition~\ref{defmoa}, namely that every possible
$t$-tuple occurs an equal number of times as a row in the $b^m \times t$ subarray formed by the columns
$\bsz_{i_1},\ldots,\bsz_{i_t}$. Take any $h_j\in R(b^{e_{i_j}})$
for $1\le j\le t$. For $1\le n\le b^m$ we have
$(z_{i_1}(n),\ldots,z_{i_t}(n))=(h_1,\ldots,h_t)$ if and only if
$z_{i_j}(n)=h_j$ for $1\le j\le t$, which is equivalent to
$\left\lfloor b^{e_{i_j}} x_n^{(i_j)}\right\rfloor= h_j$ for $1\le
j\le t$. The latter condition holds if and only if
$$
\bsx_n\in J:= \prod_{i=1}^s J_i,
$$
where
$$
J_i=
\begin{cases}
[h_jb^{-e_{i_j}},(h_j+1)b^{-e_{i_j}}) & \text{if } i=i_j \text{ for some } j \in \{1,\ldots,t\}, \\
[0,1) & \text{otherwise}.
\end{cases}
$$
Now $\lambda_s (J)=b^{-e_{i_1} -\cdots - e_{i_t}}\ge
b^{-e_{s-t+1}-\cdots - e_s}\ge b^{-m}$, and so $J$ is an
elementary interval in base $b$ to which the definition of a
$(0,m,\bse,s)$-net in base $b$ applies. Therefore the number of
integers $n$ with $1 \le n \le b^m$ such that
$(z_{i_1}(n),\ldots,z_{i_t}(n))=(h_1,\ldots,h_t)$ is given by
$$
A(J,\calP) = b^m \lambda_s(J) =b^m b^{-e_{i_1}-\cdots - e_{i_t}}
$$
for all $(h_1,\ldots, h_t)$, and the desired orthogonality
property is established.
\end{proof}

\begin{rem} \label{remlump} {\rm
 We can also combine $e_i$ that are equal, say we have $k_1$ of the $e_i$ equal to 1, $k_2$ of the $e_i$ equal to $2$, and so on up
to $k_v$ of the $e_i$ equal to $v$ with $\sum_{h=1}^v k_h=s$. Then
we obtain a mixed orthogonal array $\oa (b^m, b^{k_1}
(b^2)^{k_2}\cdots (b^v)^{k_v},t)$, where in $b^{k_1}(b^2)^{k_2}
\cdots (b^v)^{k_v}$ we delete the parts $(b^h)^{k_h}$ with
$k_h=0$. }
\end{rem}

In view of Theorem~\ref{thmnetoa}, we can apply the Rao bound for
mixed orthogonal arrays. This bound is given in~\cite[Theorem
9.4]{HSS99} and reads as follows in our notation (we again change
the $s_i$ to $l_i$ in comparison to \cite{HSS99}). The cases of
even and odd strength $t$ have to be distinguished. For binomial
coefficients, we use the standard convention ${k \choose r}=0$ for
$r > k$.

\begin{prop} \label{proprao}
The parameters of an $\oa (N, l_1^{k_1} \cdots l_v^{k_v},t)$,
where without loss of generality $l_1\le l_2\le\cdots\le l_v$,
satisfy
\begin{equation}\label{eqrao1}
N \ge \sum_{j=0}^g \sum_{I_j (v)} {k_1 \choose r_1}\cdots {k_v
\choose r_v} (l_1-1)^{r_1}\cdots (l_v-1)^{r_v}
\end{equation}
if $t=2g$, and
\begin{eqnarray}\label{eqrao2}
N &\ge& \sum_{j=0}^g \sum_{I_j (v)} {k_1 \choose r_1}\cdots {k_v \choose r_v} (l_1-1)^{r_1}\cdots (l_v-1)^{r_v}\nonumber\\
&&+\sum_{I_g (v)} {k_1 \choose r_1}\cdots {k_{v-1} \choose
r_{v-1}} {k_v -1 \choose r_v}(l_1-1)^{r_1}\cdots
(l_{v-1}-1)^{r_{v-1}} (l_v-1)^{r_v +1}
\end{eqnarray}
if $t=2g+1$, where
$$I_j (v) := \Big\{(r_1,\ldots,r_v)\in\NN_0^v: \sum_{i=1}^v r_i =j\Big\},$$
$\sum_{I_j (v)}$ denotes a sum over all $v$-tuples $(r_1,\ldots,r_v)$ in $I_j (v)$, and
$\NN_0$ is the set of nonnegative integers.
\end{prop}

We can apply the Rao bound to the mixed orthogonal arrays obtained
from $(0,m,\bse,s)$-nets. Let us start with the case where the
strength $t$ is even. We recall our assumption $e_1 \le e_2 \le
\cdots \le e_s$.

\begin{thm}\label{thmraoev}
Let $b \ge 2$ and $s \ge 2$ be integers and let $g$ be an integer
with $1 \le g \le s/2$. If there exists a $(0,m,\bse,s)$-net in
base $b$ with $m\ge e_{s-2g+1} + \cdots + e_{s-1}+ e_s$, then
necessarily
$$
\sum_{j=1}^g \sum_{1 \le i_1 < \cdots < i_j \le s} (b^{e_{i_1}}-1)
\cdots (b^{e_{i_j}}-1) \le b^m -1.
$$
\end{thm}

\begin{proof}
We apply the Rao bound in Proposition~\ref{proprao} with $N=b^m$,
strength $t=2g$, $v=s$, $k_i=1$ for $1\le i\le s$, and
$l_i=b^{e_i}$ for $1\le i\le s$. Then from~\eqref{eqrao1} we get
$$
b^m \ge \sum_{j=0}^g \sum_{I_j (s)} {1 \choose r_1}\cdots {1
\choose r_s} (l_1-1)^{r_1}\cdots (l_s-1)^{r_s}.
$$
The contribution to the outer sum over $j$ for $j=0$ is equal to
$1$. For $1 \le j \le g$, we use that ${1 \choose r} =1$ for
$r=0,1$ and ${1 \choose r} =0$ for $r \ge 2$. Hence it suffices to
restrict the sum over $I_j(s)$ to the subset
$$
\Big\{(r_1,\ldots,r_s) \in \{0,1\}^s: \sum_{i=1}^s r_i=j \Big\}.
$$
This yields the desired bound.
\end{proof}

For odd $t$, the Rao bound for mixed orthogonal arrays obtained
from $(0,m,\bse,s)$-nets attains the following form (the proof of
Theorem~\ref{thmraoodd} is similar to that of
Theorem~\ref{thmraoev}).

\begin{thm}\label{thmraoodd}
Let $b \ge 2$ and $s \ge 3$ be integers and let $g$ be an integer
with $1 \le g \le (s-1)/2$. If there exists a $(0,m,\bse,s)$-net
in base $b$ with $m\ge e_{s-2g} + \cdots + e_{s-1}+ e_s$, then
necessarily
$$
\sum_{j=1}^g \sum_{1 \le i_1 < \cdots < i_j \le s} (b^{e_{i_1}}-1)
\cdots (b^{e_{i_j}}-1) +(b^{e_s}-1) \sum_{1 \le i_1 < \cdots < i_g
\le s-1} (b^{e_{i_1}}-1) \cdots (b^{e_{i_g}}-1) \le b^m -1.
$$
\end{thm}

\begin{rem} \label{remdiff} {\rm
It is a natural question whether the Rao bound yields different
results depending on whether one lumps together identical $e_i$ or
not. The answer to this question is negative. We consider the Rao
bound in two different versions, where we distinguish the
parameters by marking them with superscripts (NL) for the case
where there is ``no lumping'' and (L) where there is ``lumping''.
To be more precise, there are two different situations: (i) the
case where we do not lump together the $e_i$ with the same
value---in this case, we count $v^{(\mathrm{NL})}=s$ values of the
$l_i^{(\mathrm{NL})}$, and $k_1^{(\mathrm{NL})}=\cdots =
k_s^{(\mathrm{NL})} =1$; (ii) the case where we do lump together
the $e_i$ with the same value---in this case, we count
$v=v^{(\mathrm{L})}\le s$ different values of the
$l_i^{(\mathrm{L})}$ and
$k_1^{(\mathrm{L})},\ldots,k_v^{(\mathrm{L})} \ge 1$ with
$\sum_{h=1}^v k_h^{(\mathrm{L})}=s$. We consider for simplicity
the case where $t$ is even and we claim that for every $u \in
\NN$, the right-hand side of the Rao bound~\eqref{eqrao1} has the
same value for the cases {\rm (i)} and {\rm (ii)}. For the proof,
we take real numbers $y_1,\ldots,y_s$, a variable $X$, and the
polynomial given by the product $\prod_{i=1}^s (1+y_iX)$. We write
this polynomial in the form
\begin{equation} \label{eqgf}
\prod_{i=1}^s \Big(\sum_{r=0}^{\infty} {1 \choose r} y_i^r X^r
\Big) = \prod_{h=1}^v (1+b_h X)^{k_h} = \prod_{h=1}^v \Big(
\sum_{r=0}^{\infty} {k_h \choose r} b_h^r X^r \Big).
\end{equation}
Here $k_h$ of the $y_i$ are equal to $b_h$ for $1 \le h \le v$ and
$\sum_{h=1}^v k_h=s$. For $j=0,1,\ldots,g$, we compare the
coefficients of $X^j$ on the leftmost and rightmost side
of~\eqref{eqgf}, then we sum over $j=0,1,\ldots,u$, and finally we
substitute $y_i=l_i^{(\mathrm{NL})}$ for $1 \le i \le s$, thus
proving the claim. }
\end{rem}

\section{Necessary conditions for $(0,\bse,s)$-sequences}\label{secnecseq}

In this section, we derive some necessary conditions for the
existence of $(0,\bse,s)$-sequences. First of all, we note that,
by using~\cite[Proposition~4]{KN14}, we obtain necessary
conditions on the parameters of $(0,\bse,s)$-sequences in base $b$
from the necessary conditions on the parameters of
$(0,m,\bse,s)$-nets in base $b$ stated in Section~\ref{secnecnet}.
However, there are further conditions that we can derive, as we
will now show.

If not stated otherwise, we assume throughout this section that,
without loss of generality, the entries $e_i$ of the $s$-tuple
$\bse \in \NN^s$ are ordered in a nondecreasing manner, i.e., the
first $k_1$ entries of $\bse$ are equal to $1$, the next $k_2$
entries of $\bse$ are equal to $2$, etc., where the $k_r$ are
nonnegative integers.

\begin{thm} \label{thmkr}
For every $(0,\bse,s)$-sequence in base $b$ for which $k_r$ of the
$e_i$ are equal to $r$ for all $r \in \NN$ and some nonnegative
integers $k_r$, we must have $k_r \le b^r$ for all $r \in \NN$.
\end{thm}

\begin{proof}
For a $k_r>0$, we consider the projection of the given sequence
onto those coordinates that correspond to the $e_i$ with $e_i=r$.
This projection yields a $(0,\bsr,k_r)$-sequence in base $b$ with
$\bsr =(r,\ldots,r) \in \NN^{k_r}$. By
using~\cite[Theorem~4]{KN14}, we obtain a $(0,k_r)$-sequence in
base $b^r$. However, it is well known from the theory of classical
$(u,s)$-sequences that a $(0,k_r)$-sequence in base $b^r$ can
exist only if $k_r\le b^r$ (see~\cite[Corollary~4.36]{DP10}
and~\cite[Corollary~4.24]{N92}). The same principle can be applied
to all $k_r>0$.
\end{proof}

\begin{rem} {\rm
The bound $k_r\le b^r$ in Theorem~\ref{thmkr} is essentially best
possible for prime powers $b$. Indeed, suppose that $b=q$ is a
prime power. We consider a Niederreiter sequence in base $q$ for
which we use all monic irreducible polynomials over the finite
field $\Field_q$ (ordered according to their degrees in a
nondecreasing manner) as the generating polynomials
(see~\cite[Section~8.1]{DP10} for the theory of Niederreiter
sequences). Then by a result of Tezuka~\cite{T13}, for every $s
\in \NN$ the $s$-dimensional version of this sequence is a
$(0,\bse,s)$-sequence in base $q$, where $\bse =(e_1,\ldots,e_s)$
with $e_i$ being the degree of the $i$th generating polynomial for
$1 \le i \le s$. On the other hand, in this case we have for every
$r\in\NN$ that $k_r=N_q (r)$, where $N_q (r)$ denotes the number
of monic irreducible polynomials over $\Field_q$ of degree $r$. It
is well known that $N_q (r)$ has the order of magnitude $q^r/r$
(see~\cite[Theorem~3.25]{LN86}), which differs only by the factor
$r$ from the upper bound $q^r$ on $k_r$. }
\end{rem}

We can extend the principle in Theorem~\ref{thmkr} further.
Suppose that we are given a $(0,\bse,s)$-sequence in base $b$ for
which $k_r\in\NN_0$ of the $e_i$ are equal to $r$ for $r\in\NN$.
Now we consider a collection of positive
$k_{r_1},k_{r_2},\ldots,k_{r_w}$, where the least common multiple
of $r_1,\ldots,r_w$ is denoted by $L$. Then by projecting onto
those coordinates corresponding to the $e_i$ that are equal to one
of the $r_1,\ldots,r_w$, we see again by~\cite[Theorem~4]{KN14}
that this projection is a $(0,k_{r_1}+\cdots +k_{r_w})$-sequence
in base $b^L$. Hence we obtain the necessary condition
$k_{r_1}+\cdots +k_{r_w}\le b^L$. In particular, if
$\mathrm{lcm}(r_1,\ldots,r_w)=r_w$, then we get $k_{r_1}+\cdots
+k_{r_w}\le b^{r_w}$ as a necessary condition. The latter
condition yields a considerable refinement of Theorem~\ref{thmkr}.

\section{Mixed ordered orthogonal arrays} \label{secooa}

We extend our findings regarding the connection between mixed
orthogonal arrays and $(u,m,\bse,s)$-nets further. It is known
that classical $(u,m,s)$-nets are closely related to the concept
of ordered orthogonal arrays, a generalization of orthogonal
arrays (see~\cite[Section~6.2]{DP10}). We now discuss an analogous
relationship between $(u,m,\bse,s)$-nets and ordered orthogonal
arrays over more than one alphabet which we call mixed ordered
orthogonal arrays.

Consider a $(u,m,\bse,s)$-net $\calP$ in base $b$ with $b\ge 2$,
$s\ge 2$, and $\bse=(e_1,\ldots,e_s)\in\NN^s$, where we again
assume without loss of generality that $e_1\le e_2\le \cdots\le
e_s$. We suppose that $m$ is an integer with $m \ge u+e_s$.

Choose positive integers $\beta_i \le \lfloor (m-u)/e_i \rfloor$
for $1 \le i \le s$. Let the points of the net $\calP$ be
$$
\bsx_n=(x_n^{(1)},\ldots,x_n^{(s)})\in [0,1)^s \quad \mbox{for }
n=1,\ldots, b^m,
$$
where
$$
x_n^{(i)}=\sum_{l=1}^m x_{l,n}^{(i)}b^{-l} \quad \mbox{for } 1\le
n\le b^m \mbox{ and } 1\le i\le s,
$$
with all $x_{l,n}^{(i)}\in R(b)$. For $1\le n\le b^m$ and $1\le
i\le s$, $1\le \rho_i\le \beta_i$, define
$$
z_{i,\rho_i} (n):=b^{\rho_i e_i}\sum_{l=(\rho_i-1)e_i +1}^{\rho_i
e_i} x_{l,n}^{(i)}b^{-l} \in R(b^{e_i}).
$$
Arrange these integers into the $b^m \times (\beta_1 +\cdots +
\beta_s)$ array
\begin{multline*}
Z= (z_{i,\rho_i} (n))_{1\le n\le b^m; 1\le i\le s, 1\le \rho_i\le \beta_i}=\\
=\begin{pmatrix}
 z_{1,1}(1) & \ldots & z_{1,\beta_1} (1) & \ldots &\ldots & z_{s,1} (1) &\ldots & z_{s,\beta_s} (1)\\
 z_{1,1}(2) & \ldots & z_{1,\beta_1} (2) & \ldots &\ldots & z_{s,1} (2) &\ldots & z_{s,\beta_s} (2)\\
\vdots & \vdots &\vdots &\vdots &\vdots &\vdots &\vdots &\vdots \\
 z_{1,1}(b^m) & \ldots & z_{1,\beta_1} (b^m) & \ldots &\ldots & z_{s,1} (b^m) &\ldots & z_{s,\beta_s} (b^m)\\
\end{pmatrix}.
\end{multline*}
Now we show the following property of this array, with an obvious notation for the columns of $Z$ (compare with
the proof of Theorem~\ref{thmnetoa}).

\begin{prop} \label{propoc}
Let $\calP$ be a $(u,m,\bse,s)$-net in base $b$ and let $Z$ be the
array obtained from $\calP$ as described above. Choose an integer
$t$ with $1\le t\le s$ and integers $1\le i_1< i_2 < \cdots <
i_t\le s$. Furthermore, choose positive integers
$\kappa_{i_1},\ldots,\kappa_{i_t}$ such that $\kappa_{i_j}\le
\beta_{i_j}$ for $ 1 \le j \le t$ and
$$
\sum_{j=1}^t \kappa_{i_j} e_{i_j}\le m-u.
$$
Then the columns
$$
\bsz_{i_1,1},\ldots,\bsz_{i_1,\kappa_{i_1}},\bsz_{i_2,1},\ldots,\bsz_{i_2,\kappa_{i_2}},\ldots\ldots,\bsz_{i_t,1},\ldots,\bsz_{i_t,\kappa_{i_t}}
$$
of the array $Z$ are orthogonal in the sense that, with
$d=\sum_{j=1}^t \kappa_{i_j}$, every possible $d$-tuple occurs an
equal number of times as a row in the $b^m \times d$ subarray of
$Z$ formed by these columns.
\end{prop}

\begin{proof}
Take any $(h_1 ^{(j)},\ldots,h_{\kappa_{i_j}}^{(j)})\in
(R(b^{e_{i_j}}))^{\kappa_{i_j}}$ for $1\le j\le t$. For $1\le n\le
b^m$ we have
\begin{equation}\label{eqcondorth}
 (z_{i_1,1}(n),\ldots,z_{i_1,\kappa_{i_1}}(n),\ldots,z_{i_t,1}(n),\ldots,z_{i_t,\kappa_{i_t}}(n))
= (h_1 ^{(1)},\ldots,h_{\kappa_{i_1}}^{(1)},\ldots,h_1
^{(t)},\ldots,h_{\kappa_{i_t}}^{(t)})
\end{equation}
if and only if $z_{i_j,\rho_{i_j}}(n)=h_{\rho_{i_j}}^{(j)}$ for
$1\le j\le t$ and $1\le\rho_{i_j}\le \kappa_{i_j}$. The latter
condition means that
$$
b^{\rho_{i_j}e_{i_j}}\sum_{l=(\rho_{i_j}-1)e_{i_j} +1}^{\rho_{i_j}
e_{i_j}} x_{l,n}^{(i_j)}b^{-l}=h_{\rho_{i_j}}^{(j)}
$$
for $1\le j\le t$ and $1\le \rho_{i_j}\le \kappa_{i_j}$, which is
equivalent to
$$\sum_{l=(\rho_{i_j}-1)e_{i_j} +1}^{\rho_{i_j} e_{i_j}} \frac{x_{l,n}^{(i_j)}}{b^l}=\frac{h_{\rho_{i_j}}^{(j)}}{b^{\rho_{i_j}e_{i_j}}}$$
for $1\le j\le t$ and $1\le \rho_{i_j}\le \kappa_{i_j}$. This is,
in turn, equivalent to
$$x_n^{(i_j)}\in \left[\sum_{\rho_{i_j}=1}^{\kappa_{i_j}}\frac{h_{\rho_{i_j}}^{(j)}}{b^{\rho_{i_j}e_{i_j}}},
\sum_{\rho_{i_j}=1}^{\kappa_{i_j}}\frac{h_{\rho_{i_j}}^{(j)}}{b^{\rho_{i_j}e_{i_j}}}+\frac{1}{b^{\kappa_{i_j}e_{i_j}}}\right)=:
\left[\frac{a^{(i_j)}}{b^{\kappa_{i_j}e_{i_j}}},\frac{a^{(i_j)}+1}{b^{\kappa_{i_j}e_{i_j}}}\right),$$
for $1\le j\le t$, for some integers $a^{(i_j)}\in
\{0,1,\ldots,b^{\kappa_{i_j} e_{i_j}}-1\}$. Thus,
\eqref{eqcondorth} is equivalent to
$$\bsx_n\in J:=\prod_{i=1}^s J_i,$$
where
$$J_i=\begin{cases}
       [0,1)&\mbox{if $i\notin\{i_1,\ldots,i_t\}$},\\
       [a^{(i_j)}/b^{\kappa_{i_j} e_{i_j}},(a^{(i_j)}+1)/b^{\kappa_{i_j} e_{i_j}})&\mbox{if $i=i_{j}$ for some $1\le j\le t$.}
      \end{cases}$$
However, the interval $J$ is an elementary interval in base $b$ of
volume
$$b^{-\kappa_{i_1} e_{i_1} - \cdots -\kappa_{i_t} e_{i_t}}\ge b^{u-m}.$$
Hence the definition of a $(u,m,\bse,s)$-net in base $b$ applies.
Therefore the number of integers $n$ with $1 \le n \le b^m$ such
that \eqref{eqcondorth} holds is given by
$$
b^m b^{-\kappa_{i_1} e_{i_1} - \cdots -\kappa_{i_t} e_{i_t}}
$$
for all $(h_1 ^{(1)},\ldots,h_{\kappa_{i_1}}^{(1)},\ldots,h_1
^{(t)},\ldots,h_{\kappa_{i_t}}^{(t)})$, and the desired
orthogonality property is established.
\end{proof}

We call the array
$$
Z= (z_{i,\rho_i} (n))_{1\le n\le b^m; 1\le i\le s, 1\le \rho_i\le
\beta_i}
$$
obtained from a $(u,m,\bse,s)$-net $\calP$ in base $b$ a
\emph{mixed ordered orthogonal array} and denote it by
\begin{equation} \label{eqmooa}
\ooa (b^m,(\beta_1,\ldots,\beta_s),l_1^1\cdots l_s^1,m-u),
\end{equation}
where $l_i=b^{e_i}$ for $1\le i\le s$. We call $m-u$ the
\emph{strength} of $Z$. The reason why we choose the
notation~\eqref{eqmooa} for $Z$ is as follows. If all $e_i=1$,
i.e., if $l_i=b$ for $1\le i\le s$, then a $(u,m,\bse,s)$-net in
base $b$ simplifies to a $(u,m,s)$-net in base $b$. In this case,
we may choose all $\beta_i$ equal to $m-u$, and then we obtain a
classical ordered orthogonal array with $b^m$ rows, $s(m-u)$
columns, and strength $m-u$ from the net. The connection between
$(u,m,s)$-nets and classical ordered orthogonal arrays is well
known (see~\cite[Section~6.2]{DP10} and~\cite{MS99}).

So far, we have shown that a $(u,m,\bse,s)$-net in base $b$ yields
a mixed ordered orthogonal array
$$
\ooa (b^m,(\beta_1,\ldots,\beta_s),l_1^1\cdots l_s^1,m-u)
$$
with $1 \le \beta_i \le \lfloor (m-u)/e_i \rfloor$ and
$l_i=b^{e_i}$ for $1 \le i \le s$. We are now going to prove that
the converse is also true.

Let $e_1,\ldots,e_s\in \NN$. Choose $\beta_i= \lfloor (m-u)/e_i
\rfloor$ for $1\le i\le s$, where $m$ and $u$ are integers with $m
\ge u+e_s$ and $u \ge 0$. Suppose now that $Z$ is a $b^m \times
(\beta_1 +\cdots + \beta_s)$ array with entries $z_{i,\rho_i}
(n)\in R(b^{e_i})$ for $1\le n\le b^m$, $1\le i\le s$, $1\le
\rho_i\le \beta_i$. Suppose furthermore that $Z$ satisfies the
following condition: for every choice of $t\in\{1,\ldots,s\}$ and
$\kappa_{i_1},\ldots,\kappa_{i_t} \in \NN$ such that
$\kappa_{i_j}\le \beta_{i_j}$ for $1 \le j \le t$ and
$$
\sum_{j=1}^t \kappa_{i_j} e_{i_j} \le m-u,
$$
the columns
$$
\bsz_{i_1,1},\ldots,\bsz_{i_1,\kappa_{i_1}},\ldots,\bsz_{i_t,1},\ldots,\bsz_{i_t,\kappa_{i_t}}
$$
of $Z$ have the property that each
$$
(h_1^{(1)},\ldots,h_{\kappa_{i_1}}^{(1)},\ldots,h_1^{(t)},\ldots,h_{\kappa_{i_t}}^{(t)})\in
(R(b^{e_{i_1}}))^{\kappa_{i_1}}\times\cdots\times
(R(b^{e_{i_t}}))^{\kappa_{i_t}}
$$
occurs with frequency
$$b^m b^{-\kappa_{i_1} e_{i_1}-\cdots - \kappa_{i_t} e_{i_t}}.$$

As we will show, $Z$ yields a $(u,m,\bse,s)$-net in base $b$.
Indeed, let $z_{i,\rho_i} (n)\in R(b^{e_i})$, where $1\le n\le
b^m$, $1\le i\le s$, $1\le \rho_i\le \beta_i$, be an entry of $Z$.
Then $z_{i,\rho_i} (n)$ has an expansion in base $b$ of the form
$$
z_{i,\rho_i}(n)=\sum_{l=0}^{e_i-1} x_{\rho_i e_i -l,n}^{(i)} b^l
=b^{\rho_i e_i}\sum_{l=(\rho_i-1)e_i +1}^{\rho_i e_i}
x_{l,n}^{(i)}b^{-l},
$$
where $x_{(\rho_i-1)e_i +1,n}^{(i)},\ldots,x_{\rho_i
e_i,n}^{(i)}\in R(b)$.

Hence from the entries $z_{i,1} (n),\ldots, z_{i,\beta_i} (n)$ we
obtain digits $x_{1,n}^{(i)},\ldots,x_{\beta_i e_i,n}^{(i)}\in
R(b)$ for all $1\le i\le s$ and $1\le n\le b^m$. We use these
digits to define
$$
x_n^{(i)}:=\sum_{l=1}^{\beta_i e_i} x_{l,n}^{(i)}b^{-l}\in [0,1)
$$
for $1\le i\le s$ and $1\le n\le b^m$. Finally, we put
$$
\bsx_n:=(x_n^{(1)},\ldots,x_n^{(s)}) \in [0,1)^s \quad \mbox{for }
1 \le n \le b^m.
$$
We claim that $\bsx_1,\ldots,\bsx_{b^m}$ form a $(u,m,\bse,s)$-net
in base $b$. We denote the point set consisting of the $\bsx_n$ by
$\calP$.

In order to verify the desired net property of $\calP$, let
$J=\prod_{i=1}^s J_i$ be an elementary interval in base $b$ for
which there exist a $t\in\{1,\ldots,s\}$ and indices
$i_1,\ldots,i_t \in \{1,\ldots,s\}$, $1\le i_1 < i_2 <\cdots <
i_t\le s$, such that
$$
J_i=\begin{cases}
       [0,1)&\mbox{if $i\notin\{i_1,\ldots,i_t\}$},\\
       [a^{(i_j)}/b^{\kappa_{i_j} e_{i_j}},(a^{(i_j)}+1)/b^{\kappa_{i_j} e_{i_j}})&\mbox{if $i=i_{j}$ for some $1\le j\le t$,}
      \end{cases}
$$
where the $a^{(i_j)}$ are integers satisfying $0\le a^{(i_j)} <
b^{\kappa_{i_j} e_{i_j}}$ for all $1\le j\le t$ and where the
$\kappa_{i_1},\ldots,\kappa_{i_t}$ are positive integers with
$$
\sum_{j=1}^t \kappa_{i_j} e_{i_j}\le m-u,
$$
that is, $\lambda_s (J) \ge b^{u-m}$. Note that the condition on
the $\kappa_{i_j}$ implies that no $\kappa_{i_j}$ exceeds
$\beta_{i_j}$. We need to show that $J$ contains exactly
$$
b^m b^{-\kappa_{i_1} e_{i_1}-\cdots - \kappa_{i_t} e_{i_t}}
$$
points of $\calP$. Suppose that $n$ is such that $\bsx_n\in J$,
i.e.,
$$x_{n}^{(i_j)}\in \left[\frac{a^{(i_j)}}{b^{\kappa_{i_j} e_{i_j}}},\frac{a^{(i_j)}+1}{b^{\kappa_{i_j} e_{i_j}}}\right)$$
for $1\le j\le t$. Since $0\le a^{(i_j)} < b^{\kappa_{i_j}
e_{i_j}}$, we can represent $a^{(i_j)}/b^{\kappa_{i_j} e_{i_j}}$
as
$$\frac{a^{(i_j)}}{b^{\kappa_{i_j} e_{i_j}}}=\sum_{\rho_{i_j}=1}^{\kappa_{i_j}}\frac{h_{\rho_{i_j}}^{(j)}}{b^{\rho_{i_j}e_{i_j}}}$$
for some $h_1^{(j)},\ldots,h_{\kappa_{i_j}}^{(j)}\in
R(b^{e_{i_j}})$. Then $\bsx_n\in J$ is equivalent to
$$x_n^{(i_j)}\in \left[\sum_{\rho_{i_j}=1}^{\kappa_{i_j}}\frac{h_{\rho_{i_j}}^{(j)}}{b^{\rho_{i_j}e_{i_j}}},
\sum_{\rho_{i_j}=1}^{\kappa_{i_j}}\frac{h_{\rho_{i_j}}^{(j)}}{b^{\rho_{i_j}e_{i_j}}}+\frac{1}{b^{\kappa_{i_j}e_{i_j}}}\right)$$
for all $j\in\{1,\ldots,t\}$. This, however, is equivalent to
$$\sum_{l=(\rho_{i_j}-1)e_{i_j} +1}^{\rho_{i_j} e_{i_j}} \frac{x_{l,n}^{(i_j)}}{b^l}=\frac{h_{\rho_{i_j}}^{(j)}}{b^{\rho_{i_j}e_{i_j}}}$$
for $1\le j\le t$ and $1\le \rho_{i_j}\le \kappa_{i_j}$, which
means that
$$
z_{i_j,\rho_{i_j}}(n)=b^{\rho_{i_j}e_{i_j}}\sum_{l=(\rho_{i_j}-1)e_{i_j}
+1}^{\rho_{i_j} e_{i_j}}
x_{l,n}^{(i_j)}b^{-l}=h_{\rho_{i_j}}^{(j)}
$$
for $1\le j\le t$ and $1\le \rho_{i_j}\le \kappa_{i_j}$. By the
orthogonality properties of the columns of $Z$ that we assumed
above, the latter condition is fulfilled for exactly
$$
b^m b^{-\kappa_{i_1} e_{i_1}-\cdots - \kappa_{i_t} e_{i_t}}
$$
indices $n$. This shows that $\calP$ is indeed a
$(u,m,\bse,s)$-net in base $b$. In summary, we have shown the
following result.

\begin{thm} \label{thmequiv}
The existence of a $(u,m,\bse,s)$-net in base $b$ is equivalent to
the existence of a mixed ordered orthogonal array
$$\ooa (b^m,(\beta_1,\ldots,\beta_s),l_1^1\cdots l_s^1,m-u)
$$
with $l_i=b^{e_i}$ and $\beta_i = \lfloor (m-u)/e_i \rfloor$ for
$1 \le i \le s$.
\end{thm}

\begin{rem} \label{remgff} {\rm
Theorem~\ref{thmequiv} can be used for the construction of mixed ordered orthogonal arrays, by starting from a known construction of a
$(u,m,\bse,s)$-net. A powerful construction of such nets was presented in \cite[Section~5]{KN14} and it employs global function fields,
that is, algebraic function fields of one variable over a finite field. We use the standard terminology for global function fields in the
monographs~\cite{NX09} and~\cite{St09}. Let $F$ be a global function field with full constant field $\FF_q$, where $q$ is an arbitrary
prime power, and let $g(F)$ be the genus of $F$. For an integer $s \ge 2$, let $P_1,\ldots,P_s$ be $s$ distinct places of $F$. Let $e_i$
be the degree of $P_i$ for $1 \le i \le s$ and put $\bse =(e_1,\ldots,e_s) \in \NN^s$. Then for every integer $m \ge \max(1,g(F))$ which 
is a multiple of ${\rm lcm}(e_1,\ldots,e_s)$, there is a construction of a $(u,m,\bse,s)$-net in base $q$ with $u=g(F)$. The condition on $m$
can be relaxed in many cases (see \cite[Remark~3]{KN14}). Mixed ordered orthogonal arrays corresponding to these nets can be read off
from Theorem~\ref{thmequiv}. }
\end{rem}

\section{A bound for mixed ordered orthogonal arrays} \label{secboundooa}

Throughout this section, let $Z$ be a mixed ordered orthogonal
array~\eqref{eqmooa} obtained from a $(u,m,\bse,s)$-net in base
$b$ according to Proposition~\ref{propoc}. We denote by $C$ the
collection of all columns of $Z$ and, for $1\le i\le s$, we define
$C_i$ to be the collection of the columns
$\bsz_{i,1},\ldots,\bsz_{i,\beta_i}$ of $Z$. We generalize the argumentation in~\cite{MS99},
which corresponds to the special case $e_i=1$ for $1 \le i \le s$.

Suppose that $D:=(D_1,\ldots,D_s)$ is an $s$-tuple of functions,
where
$$D_i: C_i\To R(b^{e_i}) \qquad \mbox{for } 1 \le i \le s.$$
For two functions $D_i^{(1)}, D_i^{(2)}$, both mapping from $C_i$
to $R (b^{e_i})$, we define $D_i^{(1)}- D_i^{(2)}$ by
$$(D_i^{(1)}-D_i^{(2)})(\bsz):=D_i^{(1)} (\bsz) - D_i^{(2)}(\bsz) \pmod {b^{e_i}}.$$

We now define two quantities that are associated with an $s$-tuple
$D=(D_1,\ldots,D_s)$ as given above. First, we define the profile
of $D=(D_1,\ldots,D_s)$ by
$$\prof (D) = \prof ((D_1,\ldots,D_s)):=(d_1,\ldots,d_s),$$
where
$$
d_i=\begin{cases}
       0 & \mbox{if $D_i (\bsz_{i,\rho_i})=0$ for $1\le \rho_i\le \beta_i$,}\\
       \max\{\rho_i: D_i (\bsz_{i,\rho_i})\neq 0\} & \mbox{otherwise,}
      \end{cases}
$$
for $1\le i\le s$. Note that $0\le d_i\le \beta_i$ for $1\le i\le
s$. Furthermore, we define the height of $D=(D_1,\ldots,D_s)$ as
$$
\hei (D)=\hei ((D_1,\ldots,D_s)):= \sum_{i=1}^s d_i e_i.
$$
Moreover, note that if $Z$ is a mixed ordered orthogonal
array~\eqref{eqmooa} obtained from a $(u,m,\bse,s)$-net in base
$b$ according to Proposition~\ref{propoc} and if
$$
\hei ((D_1,\ldots,D_s))=\sum_{i=1}^s d_i e_i \le m-u,
$$
then the columns
$$
\bsz_{1,1},\ldots,\bsz_{1,\delta_1},\ldots,\bsz_{s,1},\ldots,\bsz_{s,\delta_s}
$$
are orthogonal for all $\delta_j\le d_j$, $1\le j\le s$, by
Proposition~\ref{propoc}.

We now show the following theorem which is the ``mixed"
analog of~\cite[Lemma~3.1]{MS99}. This theorem gives a necessary condition on 
the parameters of a mixed ordered orthogonal array.

\begin{thm} \label{thmcalD}
Let $Z$ be a mixed ordered orthogonal array~\eqref{eqmooa}
obtained from a $(u,m,\bse,s)$-net in base $b$. Let $\mathcal{D}$
be a set of functions defined on $C$ such that
$$
\hei
((D_1^{(1)},\ldots,D_s^{(1)})-(D_1^{(2)},\ldots,D_s^{(2)}))\le m-u
$$
for all $(D_1^{(1)},\ldots,D_s^{(1)}),
(D_1^{(2)},\ldots,D_s^{(2)})\in\mathcal{D}$. Then $b^m\ge
\abs{\mathcal{D}}$.
\end{thm}

\begin{proof}
Let $\omega_j:=e^{2\pi\icomp / b^{e_j}} \in \CC$ and let
$1,\omega_j,\omega_j^2,\ldots,\omega_j^{b^{e_j}-1}$ be the
$b^{e_j}$-th roots of unity for $1\le j\le s$. Suppose now that
$Z$ is as in the theorem. Let $C$ and $C_i$, $1 \le i \le s$, be as
in the beginning of this section. We can identify a column
$\bsc\in C_i$ with a vector $v_{\bsc}$ over the alphabet
$1,\omega_i,\ldots,\omega_i^{b^{e_i}-1}$, that is,
$v_{\bsc}\in\CC^{b^m}$.

Let $D=(D_1,\ldots,D_s)\in\mathcal{D}$, where $D_i:C_i\To R
(b^{e_i})$ for $1 \le i \le s$. For every $\bsc\in C$, we can
identify a unique $i\in\{1,\ldots,s\}$ such that $\bsc\in C_i$,
and we take $D_i (\bsc)$ copies of the corresponding
$v_{\bsc}\in\CC^{b^m}$. We repeat this procedure for each $\bsc\in
C$ and we obtain
$$\sum_{i=1}^s \sum_{\rho_i=1}^{\beta_i} D_i (\bsz_{i,\rho_i}) $$
vectors in $\CC^{b^m}$. We then take the componentwise product of
these vectors and thereby obtain a vector $v_D\in\CC^{b^m}$
determined by $D$. This vector is of the form
$$
\begin{pmatrix}
   \prod_{i=1}^s \prod_{\rho_i =1}^{\beta_i}\omega_i^{k_{i,\rho_i}^{(1)} D_i (\bsz_{i,\rho_i})}\\
   \vdots\\
   \prod_{i=1}^s \prod_{\rho_i =1}^{\beta_i}\omega_i^{k_{i,\rho_i}^{(b^m)} D_i (\bsz_{i,\rho_i})}
  \end{pmatrix}
$$
with the $k_{i,\rho_i}^{(n)}$ being elements of $R(b^{e_i})$ for
$1\le i\le s$ and $1\le \rho_i\le \beta_i$. For two distinct
elements $D^{(1)}=(D_1^{(1)},\ldots, D_s^{(1)})$ and
$D^{(2)}=(D_1^{(2)},\ldots,D_s^{(2)})$ of $\mathcal{D}$, we have
by assumption,
$$\hei (D^{(1)}-D^{(2)})\le m-u.$$
For short, we write $E:=D^{(1)}-D^{(2)}$, with
$E_i=D_i^{(1)}-D_i^{(2)}$ for $1\le i\le s$. Hence we know that
$\hei (E)\le m-u$. Thus, there exist integers $d_1,\ldots, d_s$
with $0\le d_i\le \beta_i$ for $1 \le i \le s$ such that
$\sum_{i=1}^s d_i e_i\le m-u$ and $E_i (\bsz_{i,\rho_i})=0$ for
$\rho_i> d_i$. Formulating this property of $E$ slightly
differently, we can say that there exist positive integers
$\kappa_{i_1},\ldots,\kappa_{i_t}$ with $1\le i_1 < \cdots < i_t
\le s$ and $\kappa_{i_j} \le \beta_{i_j}$ for $1 \le j \le t$ such
that
$$\sum_{j=1}^t \kappa_{i_j} e_{i_j} \le m-u$$
as well as $E_i(\bsz_{i,\rho_i})>0$ if and only if $i=i_j$ for
some $j$ and $\rho_{i_j}\le \kappa_{i_j}$. By
Proposition~\ref{propoc}, the columns
$$\bsz_{i_1,1},\ldots,\bsz_{i_1,\kappa_{i_1}},\ldots,\bsz_{i_t,1},\ldots,\bsz_{i_t,\kappa_{i_t}}$$
are orthogonal, and so also the $v_{\bsc}$ corresponding to these
columns of $Z$ are orthogonal and each possible combination of
symbols occurs with frequency $f=b^m b^{-\kappa_{i_1}
e_{i_1}-\cdots - \kappa_{i_t} e_{i_t}}$. Let now $\langle
\cdot,\cdot\rangle$ denote the usual Hermitian inner product in
$\CC^{b^m}$. We study the expression
\begin{eqnarray*}
\langle v_{D^{(1)}},v_{D^{(2)}}\rangle &=& \sum_{n=1}^{b^m} \prod_{i=1}^s \prod_{\rho_i =1}^{\beta_i} \omega_i^{k_{i,\rho_i}^{(n)} E_i (\bsz_{i,\rho_i})}\\
&=&\sum_{n=1}^{b^m} \prod_{j=1}^t \prod_{\rho_{i_j}
=1}^{\kappa_{i_j}} \omega_{i_j}^{k_{i_j,\rho_{i_j}}^{(n)} E_{i_j}
(\bsz_{i_j,\rho_{i_j}})}.
\end{eqnarray*}
Due to the above-mentioned orthogonality properties of the
$v_{\bsc}$, we can write
\begin{eqnarray*}
\langle v_{D^{(1)}},v_{D^{(2)}}\rangle&=& f
\sum_{k_{i_1,1}=0}^{b^{e_{i_1}}-1}\cdots
\sum_{k_{i_1,\kappa_{i_1}}=0}^{b^{e_{i_1}}-1} \cdots
\sum_{k_{i_t,1}=0}^{b^{e_{i_t}}-1}\cdots
\sum_{k_{i_t,\kappa_{i_t}}=0}^{b^{e_{i_t}}-1}
\prod_{j=1}^t \prod_{\rho_{i_j} =1}^{\kappa_{i_j}} \omega_{i_j}^{k_{i_j,\rho_{i_j}} E_{i_j} (\bsz_{i_j,\rho_{i_j}})}\\
&=& f \prod_{j=1}^t \prod_{\rho_{i_j} =1}^{\kappa_{i_j}}
\sum_{k_{i_j,\rho_{i_j}}=0}^{b^{e_{i_j}}-1}
\omega_{i_j}^{k_{i_j,\rho_{i_j}} E_{i_j} (\bsz_{i_j,\rho_{i_j}})}.
\end{eqnarray*}
However, as $E_{i_j} (\bsz_{i_j,\rho_{i_j}}) \not\equiv 0 \ ({\rm
mod} \ b^{e_{i_j}})$ in the last sum, it is clear that
$$ \sum_{k_{i_j,\rho_{i_j}}=0}^{b^{e_{i_j}}-1}\omega_{i_j}^{k_{i_j,\rho_{i_j}} E_{i_j} (\bsz_{i_j,\rho_{i_j}})}=
\sum_{k_{i_j,\rho_{i_j}}=0}^{b^{e_{i_j}}-1}\left(\omega_{i_j}^{E_{i_j}
(\bsz_{i_j,\rho_{i_j}})}\right)^{k_{i_j,\rho_{i_j}}}=0.$$ We
therefore see that the collection of the $v_D$ with
$D\in\mathcal{D}$ is orthogonal with respect to
$\langle\cdot,\cdot\rangle$, and therefore $\{v_D:
D\in\mathcal{D}\}$ is a linearly independent set of vectors in
$\CC^{b^m}$. This implies the desired result.
\end{proof}

\begin{rem}\rm
A natural question is whether one can derive effective concrete bounds on the $u$-value of $(u,m,\bse,s)$-nets in base $b$ from Theorem~\ref{thmcalD}, 
as it was done analogously for ordinary $(u,m,s)$-nets in \cite{MS99}. However, this question appears to be very challenging, and is therefore left open 
for future research.
\end{rem}

\begin{small}
\noindent \textbf{Authors' addresses:}\\ \\
\noindent Peter Kritzer\\ 
Department of Financial Mathematics and Applied Number Theory,\\ 
Johannes Kepler University Linz,\\ 
Altenbergerstr. 69, A-4040 Linz, AUSTRIA.\\
\texttt{peter.kritzer@jku.at}\\ 

\noindent Harald Niederreiter\\
Johann Radon Institute for Computational and Applied Mathematics,\\
Austrian Academy of Sciences,\\
Altenbergerstr. 69, A-4040 Linz, AUSTRIA,\\
and\\
Department of Mathematics,\\ 
University of Salzburg,\\ 
Hellbrunnerstr. 34, A-5020 Salzburg, AUSTRIA,\\
\texttt{ghnied@gmail.com}
\end{small}


\begin{thebibliography}{00}

\bibitem{DP10}
J.~Dick, F.~Pillichshammer. \textit{Digital Nets and Sequences:
Discrepancy Theory and Quasi-Monte Carlo Integration}. Cambridge
University Press, Cambridge, 2010.

\bibitem{HSS99}
A.S.~Hedayat, N.J.A.~Sloane, J.~Stufken. \textit{Orthogonal
Arrays: Theory and Applications}. Springer, New York, 1999.

\bibitem{H15}
R.~Hofer. Generalized Hofer-Niederreiter sequences and their
discrepancy from a $(\boldsymbol{U},\bse,s)$-point of view. J.
Complexity 31, 260--276, 2015.

\bibitem{HN13}
R.~Hofer, H.~Niederreiter. A construction of $(t,s)$-sequences
with finite-row generating matrices using global function fields.
Finite Fields Appl. 21, 97--110, 2013.

\bibitem{KN14}
P.~Kritzer, H.~Niederreiter. Propagation rules for
$(u,m,\bse,s)$-nets and $(u,\bse,s)$-sequences. J. Complexity 31,
457--473, 2015.

\bibitem{L96}
K.M.~Lawrence. A combinatorial characterization of $(t,m,s)$-nets
in base $b$. J. Combinatorial Designs 4, 275--293, 1996.

\bibitem{LM98}
C.F.~Laywine, G.L.~Mullen. \textit{Discrete Mathematics Using Latin Squares}.
Wiley, New York, 1998.

\bibitem{LP14}
G.~Leobacher, F.~Pillichshammer. \textit{Introduction to
Quasi-Monte Carlo Integration and Applications}. Birkh\"auser and
Springer International, Heidelberg, 2014.

\bibitem{LN86}
R.~Lidl, H.~Niederreiter. \textit{Introduction to Finite Fields
and Their Applications}, revised edition. Cambridge University
Press, Cambridge, 1994.

\bibitem{MS99}
W.J.~Martin, D.R.~Stinson. A generalized Rao bound for ordered
orthogonal arrays and $(t,m,s)$-nets. Canad. Math. Bull. 42,
359--370, 1999.

\bibitem{MS96}
G.L.~Mullen, W.Ch.~Schmid. An equivalence between $(t,m,s)$-nets
and strongly orthogonal hypercubes. J. Combinatorial Theory Ser. A
76, 164--174, 1996.

\bibitem{N87}
H.~Niederreiter. Point sets and sequences with small discrepancy.
Monatsh. Math. 104, 273--337, 1987.

\bibitem{N92}
H.~Niederreiter. \textit{Random Number Generation and Quasi-Monte
Carlo Methods}. SIAM, Philadelphia, 1992.

\bibitem{N13}
H.~Niederreiter. $(t,m,s)$-nets and $(t,s)$-sequences.
\textit{Handbook of Finite Fields} (G.L.~Mullen, D.~Panario,
eds.), pp. 619--630, CRC Press, Boca Raton, FL, 2013.

\bibitem{NX01}
H.~Niederreiter, C.P.~Xing. \textit{Rational Points on Curves over
Finite Fields: Theory and Applications}. Cambridge University
Press, Cambridge, 2001.

\bibitem{NX09}
H.~Niederreiter, C.P.~Xing. \textit{Algebraic Geometry in Coding Theory and
Cryptography}. Princeton University Press, Princeton, NJ, 2009.

\bibitem{NY13}
H.~Niederreiter, A.S.J.~Yeo. Halton-type sequences from global
function fields. Science China Math. 56, 1467--1476, 2013.

\bibitem{SS09}
R.~Sch\"urer, W.Ch.~Schmid. MinT -- new features and new results.
\textit{Monte Carlo and Quasi-Monte Carlo Methods 2008}
(P.~L'Ecuyer, A.B.~Owen, eds.), pp. 171--189, Springer, Berlin,
2009.

\bibitem{St09}
H.~Stichtenoth. \textit{Algebraic Function Fields and Codes}, second edition.
Springer, Berlin, 2009.

\bibitem{T13}
S.~Tezuka. On the discrepancy of generalized Niederreiter
sequences. J.~Complexity 29, 240--247, 2013.
\end{thebibliography}
\end{document}